\documentclass[12pt]{amsart}
\usepackage{tikz}
\usetikzlibrary{cd}
\usepackage{amsmath}
\usepackage{amssymb}
\usepackage{amsfonts}
\usepackage{stmaryrd}
\usepackage{mathrsfs}
\usepackage{graphicx}
\usepackage[font=small,labelfont=bf]{subcaption}
\usepackage{lineno}
\usepackage[a4paper,left=1cm,right=1cm,top=1.5cm,bottom=1.5cm]{geometry}
\usepackage{hyperref}
\newcommand{\defeq}{\overset{\text{d}}{=}}
\newcommand{\mh}{\mathcal{H}}
\newcommand{\mg}{\mathcal{G}}

\newcommand{\mb}{\mathcal{B}}
\newcommand{\mv}{\mathcal{V}}
\newcommand{\mi}{\mathcal{I}}

\newcommand{\bdd}{\mbox{$\partial$}}

\newcommand{\mr}{\mathscr{R}}
\newcommand{\dd}{\operatorname{Diff}^{+}}
\newcommand{\gp}{Goeritz\;group}

\begin{document}

\newtheorem{thm}{Theorem}[section]
\newtheorem{lem}[thm]{Lemma}
\newtheorem{cor}[thm]{Corollary}
\newtheorem{pro}[thm]{Proposition}
\newtheorem{exm}[thm]{Example}
\newtheorem{clm}[thm]{Claim}

\theoremstyle{definition}
\newtheorem{defn}{Definition}[section]
\newtheorem{assum}[defn]{Assumption}
\newtheorem{ques}[defn]{Question}
\newtheorem{note}[defn]{Note}

\theoremstyle{remark}
\newtheorem{rmk}{Remark}[section]

\def\square{\hfill${\vcenter{\vbox{\hrule height.4pt \hbox{\vrule
width.4pt height7pt \kern7pt \vrule width.4pt} \hrule height.4pt}}}$}
\def\T{\mathcal T}

\newenvironment{pf}{{\it Proof:}\quad}{\square \vskip 12pt}

\title{Genus three Goeritz groups of connected sums of two Lens spaces}

%%%%%%%%%%%%%%%%%%%%%%%%%%%%%%%%%%%%%%5555
\author{Hao Chen}

\address{Department of Mathematics, East China Normal University, Dongchuan Road No.500, Minhang District, Shanghai, China, 200241
} \email{hchen@stu.ecnu.edu.cn}

\author{Yanqing Zou}
\address{Department of Mathematics, East China Normal University, Dongchuan Road No.500, Minhang District, Shanghai, China, 200241
} \email{yqzou@math.ecnu.edu.cn}
%%%%%%%%%%%%%%%%%%%%%%%%%%%%%%%%%%%%%%%%%%%5

\begin{abstract}
We prove that the mapping class groups of the genus $3$ \emph{Heegaard splittings} of the connected sum of two lens spaces are finitely generated, and the corresponding reducing sphere complexes are all connected.

\end{abstract}

\maketitle

\vspace*{0.5cm} {\bf Keywords}: \emph{Heegaard splitting}, \emph{\gp}, \emph{Reducing sphere complex}.\vspace*{0.5cm}


\section{Introduction}
\label{sec1}
It is well known that every closed orientable 3-manifold $N$ admits a \emph{Heegaard splitting} $V\cup_{\Sigma} W$, which is a decomposition of $N$ into two handlebodies $V\text{ and }W$ of the same genus. Their common boundary surface $\Sigma$ is called the $Heegaard\; surface$, and the genus of $\Sigma$ is referred to as the genus of the $Heegaard\; splitting$. Two \emph{Heegaard splittings} of $N$ are said to be isotopic if their corresponding $Heegaard\;surfaces$ are ambient isotopic. The $Goeritz\;group$ $\mg\left(N,\Sigma\right)$, firstly introduced by Goeritz \cite{goeritz_abbildungen_1933}, is the group of isotopy classes of orientation-preserving diffeomorphisms of $N$ preserving these two handlebodies of the splitting setwise. As a subgroup of the mapping class group of $\Sigma$, there is an open question about it in \cite{gordon_problems_2007}. 
 %Furthermore, if any two same genera $Heegaard\; splittings$ of $N$ are isotopic, we say its $\hs$ are standard. 
\begin{ques}
  Is \emph{Goeritz group}  
 finte or finitely generated? 
\end{ques}

By works of Namazi,  Johnson,  Qiu, and the second author,  Goeritz groups of almost all distance at least 2 Heegaard splittings are finite. However, if  $V\cup_{\Sigma} W$ is weakly reducible, or equivalently, has the distance of at most 1, $\mg\left(N,\Sigma\right)$ is infinite as shown in Namazi's construction. In this paper, we focus on studying the finitely generated problem for the Goeritz group of weakly reducible Heegaard splittings.

We begin by considering the case of a reducible Heegaard splitting. A Heegaard splitting $V\cup_{\Sigma} W$ is reducible if there is a 2-sphere $S \subset N$ intersecting $\Sigma$ transversely in one essential simple  closed curve; such a 2-sphere $S$ is called a reducing sphere for $\Sigma$. Usually, we can choose a separating reducing 2-sphere $S$. Then $V\cup_{S} W$ is a connected sum of two smaller genera Heegaard splittings, denoted by $V\cup_{\Sigma}W= \left(N_1=V_{1}\cup_{\Sigma_{1}}W_{1}\right) \sharp \left(N_2=V_{2}\cup_{\Sigma_{2}}W_{2}\right)$. It is known that minimal Heegaard genera are additive under connected sums. A natural question arises:  If both of $\mg(N_i,\Sigma_i)$ are finitely generated for $i=1,2$, is $\mg\left(N,\Sigma\right)$ also finitely generated? 

If $g(\Sigma_{1})+g(\Sigma_{2})=2$, Cho and Koda proved that $\mg\left(N,\Sigma\right)$ is finitely presented \cite{cho_mapping_2019}. Here we study the case that $g(\Sigma_{1})+g(\Sigma_{2})=3$ and give an answer to Question 1.1 as follows. 

\begin{thm}\label{finitely}
If $N=V\cup_{\Sigma} W$ is a genus three reducible Heegaard splitting for a connected sum of two lens spaces,  then $\mg\left(N,\Sigma\right)$ is finitely generated.
\end{thm}

Furthermore, we study the reducing sphere complex $\mr$ for $V\cup_{\Sigma} W$, which is a subcomplex of the curve complex spanned by those curves that bound disks in both  handlebodies.
%when $g(\Sigma)=2$, a slight modification is necessary, in which two reducing curves of geometric intersection number four determine an edge of $\mr$. 
As a corollary, we have
\begin{cor}\label{reducing}
Under the same condition in Theorem \ref{finitely}, $\mr$ is connected.
\end{cor}
For any reducible Heegaard splitting $V\cup_{\Sigma}W= V_{1}\cup_{\Sigma_{1}}W_{1} \sharp V_{2}\cup_{\Sigma_{2}}W_{2}$, let $\mu=S\cap \Sigma$ be the intersection of the reducing sphere $S$ and the Heegaard surface $\Sigma$.  Although the method in the proof of Theorem 1.1 does not apply in general, it provides insight into  the
widely studied subgroup  $G_{\mu}\leqslant\mg(N,\Sigma)$, the stabilizer of $\mu$, which is a key subgroup of $\mathcal{G}(N, \Sigma)$. By its definition, it is not hard to see that there is a natural homomorphism from $G_\mu$ to $\mg(N_i,\Sigma_i)$ for each $i$. Thus it is of interest to determine whether $G_{\mu}$ is finitely generated(or finitely presented)  when both of those two Georitz groups are finitely generated(or finitely presented). Using standard combinatorial techniques,  we obtain the following result.
\begin{thm} \label{thm2}
If $\mathcal{G}(N_{1}, \Sigma_{1})$ and $\mathcal {G}(N_{2}, \Sigma_{2})$ are both finitely generated(or finitely presented), then so is $G_{\mu}$.
\end{thm}

%In the case $N=S^3$,the group $\mg_g(S^3)$ is classically known as the genus g  $\emph{\gp}$\cite{goeritz_abbildungen_1933}. $\mg_2(S^3)$ has been proven to be finitely presented and an explicit presentation given in \cite{scharlemann_automorphisms_2003,akbas_presentation_2008,cho_homeomorphisms_2008}.In addition,Powell conjecture\cite{powell_homeomorphisms_1980}  offers a finite generating set for $\mg_g(S^3)$,which has been confirmed for $g\leqslant3$.

%For lens spaces,whose $genus 2\;Goeritz\;groups$ have been proven to be finitely presented by Cho and Koda in \cite{cho_genus-two_2013,cho_connected_2016,cho_mapping_2019}; on the other hand, Cho and Koda\cite{cho_haken_2018} also show that there are many lens spaces whose $genus 2\;\hs$ are of disconnected reducing sphere complexes.

%In this paper,we study the $genus 3\;\emph{\gp}$ of $N$ and show that it is finitely generated.On the other hand,we prove that the reducing sphere complex for $genus 3\;Heegaard\; splitting$ of $N$ is connected.Prior to this,the similar conclusions for genus $2$ case have been demonstrated by Cho-Koda\cite{cho_disk_2015} and Lei\cite{lei_haken_2005}. 

%Throughout the paper,we will work in the smooth category. For ease of natations, we often do not distinguish curves, disks, spheres and  diffeomorphisms from their isotopy classes in their notation.Any curves in a surface and any surfaces in a $3$-manifold are always assumed to be properly embedded, and their intersection is transverse and minimal up to isotopy, unless otherwise mentioned.

 This paper is organized as follows. We introduce some notations in Section \ref{sec2} and study two classes of automorphisms, \emph{eyeglass twist} and \emph{visional bubble move}, in Section 3. And then we prove Theorem  \ref{finitely} and \ref{thm2} in Section 5.

{\bf Acknowledgement.} We would like to thank Professor  Ruifeng Qiu and Chao Wang for many helpful suggestions. This work was partially supported by NSFC 12131009, 12471065 and in part by Science and Technology Commission of Shanghai Municipality (No. 22DZ2229014).

\section{Preliminaries}
\label{sec2}
Throughout the paper, we denote the isotopy class of a curve $\mu$(resp. a diffeomorphism $f$) by $\bar{\mu}$(resp. $\bar{f}$). From now on, we assume that $N$ is the connected sum of two lens spaces unless otherwise specified, and that $V\bigcup_{\Sigma}W$  a genus 3 \emph{Heegaard splitting} of $N$. Let $\dd(N,\Sigma)$  be a subgroup of $\dd(N)$ defined as follows:
 $$\dd(N,\Sigma)\defeq\left\{f\in\dd(N): f(\Sigma)=\Sigma\text{ and }f\text{ preserves the orientation of }\Sigma\right\}.$$ It is well known that if an orientation-preserving diffeomorphism of $N$ preserves the \emph{Heegaard splitting} of $N$, it must preserves the orientation of $\Sigma$. Hence, the natural homomorphism $\rho_{1}: \dd(N,\Sigma)\rightarrow\mg\left(N,\Sigma\right)$ is an epimorphism.
\begin{defn}
	Two reducing spheres $S_1,S_2$(for $\Sigma$) are isotopic if there is an isotopy $$H_t(0\leq t\leq 1):\left(N,\Sigma\right)\rightarrow\left(N,\Sigma\right)$$ 
	so that $H_0=\mathrm{id}\text{ and } H_1 \left(S_1\right)=S_2$.
\end{defn}
\begin{defn}
	A triplet $\T=\left(S_1,S_2,S_3\right)$ of pairwise non-isotopic reducing spheres for $\Sigma$ is called a sphere triplet(for $\Sigma$), and spheres $S_i$ are called the  components of $\T$. We say the triplet is complete if the reducing spheres are pairwise disjoint.
\end{defn}
\begin{defn}
	Two sphere triplets $\T_1,\T_2$ are isotopic if there is an isotopy $$H_t(0\leq t\leq 1):\left(N,\Sigma\right)\rightarrow\left(N,\Sigma\right)$$ 
	so that $H_0=\mathrm{id}\text{ and } H_1 \left(\T_1\right)=\T_2$.
\end{defn}
\begin{note}
	We usually make no notational distinction between triplets and their isotopy classes when the context is clear.
\end{note}

\begin{defn}
	Two sphere triplets $\T_1,\T_2$ are congruent if they differ  by a permutation. For instance,  $\left(S_1,S_2,S_3\right)$ is congruent to $\left(S_3,S_1,S_2\right)$.
\end{defn}

We designate a complete sphere triplet $\T=\left(S_1,S_2,S_3\right)$(for $\Sigma$), as depicted in Figure \ref{hf}, such that (1) $S_i \left(i=1,2\right)$ cuts off a genus $1\;Heegaard\; splitting$ of $M_i\setminus B^3$; (2) $S_3$ cuts off a genus $1\;Heegaard\; splitting$ of a 3-ball. Clearly, $S_1\text{ and }S_2$ are two reducible  2-spheres and cobound $S^2\times I$ in $N$. We also denote $ \mu_i=S_i\cap \Sigma$, for $i=1,2,3$.
\begin{lem}\label{transitive}
   Up to congruences, $\mg\left(N, \Sigma\right)$  acts transitively on the set of isotopy classes of complete sphere triplets for $\Sigma$. 
\end{lem}
\begin{proof}
    It suffices to show that for any given triplet $\T'=\left(S'_1,S'_2,S'_3\right)$, there exists a diffeomorphism $h\in\dd(N,\Sigma)$ so that $h(\T)$ is congruent to $\T'$. Since $\bigcup_{i=1}^{3} S'_i$ divides $N$ into four regions, each of the same diffeomorphism type as the four regions divided by $\bigcup_{i=1}^{3} S_i$, we can glue all diffeomorphisms of these four regions to obtain $h$. 
\end{proof}

Using similar arguments, we can prove the following lemma.
\begin{lem}\label{stab}
    If $S$ is a common component of these two complete sphere triplets $\T$ and $\T'$, then there exists a diffeomorphism $h\in\dd(N,\Sigma)$ so that $h(\T)=\T'$ and $h(S)=S$.
\end{lem}
\begin{figure}[h]
	\includegraphics[width=0.5\textwidth]{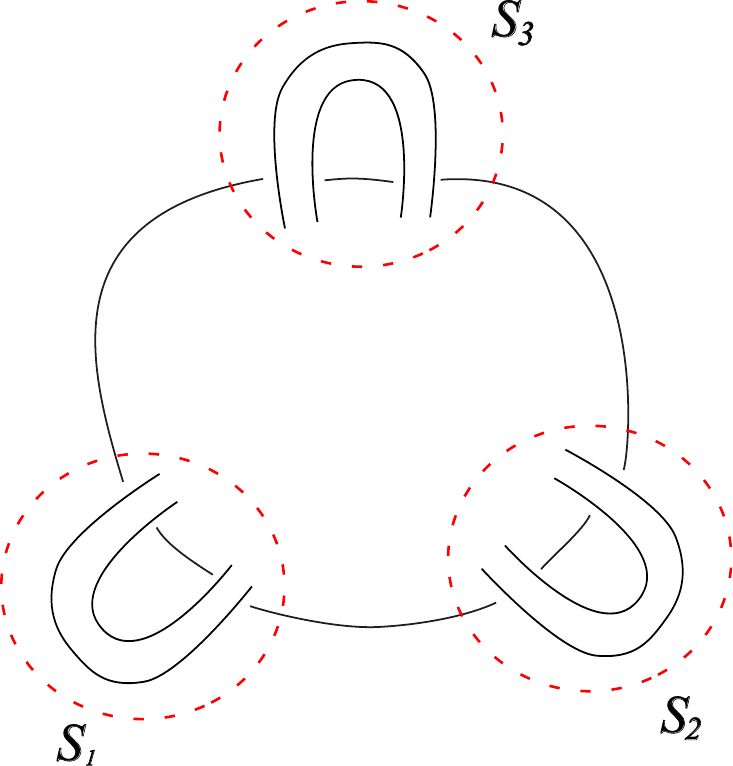}
	\caption{Heegaard surface $\Sigma$ and triplet $\T$}\label{hf}
\end{figure}

\section{Eyeglass Twist and Bubble Move}
\label{sec3}
%Suppose $N=A\cup_{\Sigma}B$ is a \emph{Heegaard splitting}. When we want to describe an element of the \emph{Goeritz group} $\mg\left(N,\Sigma\right)$, in addition to constructing a specific diffeomorphism, an alternative way is to give a loop of embeddings of $\Sigma\text{ in }N$, that is a diffeotopy $H_{t}(0\leq t\leq 1): N\rightarrow N$ so that  $H_{0}=\mathrm{id}\text{ and }H_{1}(\Sigma)=\Sigma$; the diffeomorphism $H_{1}$ exactly determines an element of the \emph{Goeritz group}. The advantage of the point of view is that we can view a diffeomorphism quite vividly. A deep treatment for the view can be found in \cite{johnson_space_2013}. In the section \ref{sec3} and the following section \ref{sec4}, we will adopt this way to define two classes of automorphisms, \emph{eyeglass twist} and \emph{bubble move}.
%The concept ``eyeglass twist" traces its origins to the Powell conjecture introduced in \cite{powell_homeomorphisms_1980}, which posits that $\mg_{g}\left(S^3\right)$ can be generated by five elements. One of these five has  been generalized by Freedman and Scharlemann into an automorphism for a weakly reducible \emph{Heegaard splitting}, which they term the eyeglass twist in \cite{freedman_powell_2018}. Below we recall their definition.

A \emph{Heegaard splitting} $N=A\cup_{\Sigma}B$ is weakly reducible if there are two disjoint properly embedded essential disks, $a\text{ and } b$ in $A\text{ and }B$ respectively. We call $(a,b)$ a weakly reducing pair for $\Sigma$. An eyeglass is a triple $(a,b, \lambda)$,  where $\left(a,b\right)$ is a weakly reducing pair for $\Sigma$ and $\lambda\subset \Sigma$  is an arc connecting $a$ and $b$ with its interior disjoint from them.  For  such a triple $\left(a,b,\lambda\right)$, we refer to $\left(a,b\right)$ as the lenses of $\eta$ and $\lambda$ as the bridge of $\eta$. Given a normal direction $\vec{n}$ pointing toward the interior of $B$, we can push the 1-handle $a\times I$ around the circumference of disk $b$ in a counterclockwise direction as in Figure \ref{et}. Actually, we have described an excursion of the handlebody $A$ that ends at initial position. More formally, an eyeglass $\eta$ defines a natural  automorphism $T_{\eta}: (N, \Sigma) \rightarrow(N, \Sigma)$, known as the \emph{positive eyeglass twist}. The inverse of this operation, which involves a clockwise excursion of $A$, is referred to as the \emph{negative eyeglass twist} and denoted $T_{\overline{\eta}}$. For an eyeglass twist $T_\eta$, the eyeglass $\eta$ is referred to as its base eyeglass. It is not hard to see that an eyeglass twist preserves the isotopy classes of its lenses. 
\begin{note}
	The above definition does not depend on the order of the two lenses of an eyeglass. In other words, if $\eta=(a,b,\lambda)$ and $\eta'=(b,a,\lambda)$, then we have  $T_\eta=T_{\eta'}$.
\end{note}
\begin{figure}[ht]
	\includegraphics[width=0.5\textwidth]{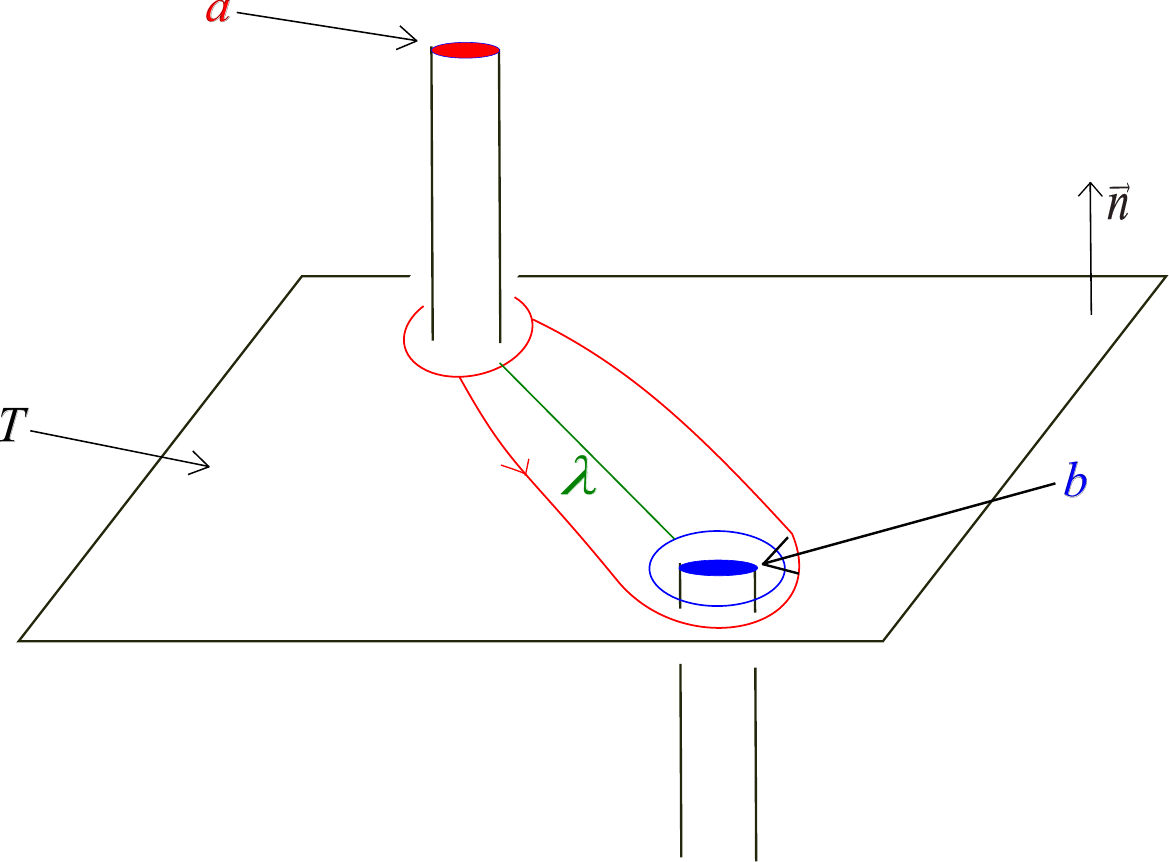}
	\caption{eyeglass twist}\label{et}
\end{figure}

Let $\eta=\left(a,b,\lambda\right)$ be an eyeglass, where $\alpha\defeq\partial a,\beta\defeq\partial b$, and $\Delta$ a regular neighborhood of $\alpha \cup \lambda \cup \beta$ in surface $\Sigma$. Denote by $\gamma$ the component(as in Figure \ref{eyeglass}.) of $\partial \Delta$, which is isotopic to neither $\alpha$ nor $\beta$. We can now describe the above situation as follows:
% We call $\gamma$ the surrogate of $\eta$, and denote it by $\mathfrak{su}(\eta)$.
\begin{center}
	$T_\eta=\tau_\alpha\cdot\tau_\beta\cdot\tau_\gamma^{-1}$,
\end{center}
where $\tau_{[\cdot]}$ denotes the \emph{left-handed Dehn twist}; see more details in \cite[Lemma 2.5]{zupan_powell_2020}.
\begin{rmk}
	Although different choices of regular neighborhoods of the eyeglass $\eta$ yield different eyeglass twists, they are all equivalent up to isotopy. Therefore, for the eyeglass $\eta$, we obtain two eyeglass twists $T_{\eta},T_{\overline{\eta}}\in \mathcal{G}(N,\Sigma)$. 
\end{rmk}
%For convenience, we sometimes also denote an eyeglass $\eta=\left(a,b,\lambda\right)$ by triple $\left(\alpha,\beta,\lambda\right)$, the corresponding eyeglass twist by $T_{(\alpha,\beta,\gamma)} $ or $T_{(a,b,\gamma)}$ when the eyeglass $\eta$ is clear from the context.
\begin{figure}[h]
	\includegraphics[width=0.5\textwidth]{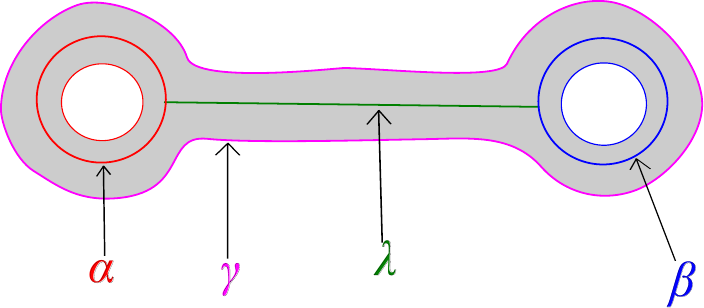}
	\caption{regular neighborhood $\Delta$}\label{eyeglass}
\end{figure}
\begin{defn}
	Suppose $\eta_1\text{ and }\eta_2$ are two eyeglasses  in $N$. They are isotopic if there is an isotopy $H_{t}(0\leq t\leq 1): (N,\Sigma)\rightarrow (N,\Sigma)$ such that $H_{0}=\mathrm{id} \text{ and } H_{1}(\eta_1)=\eta_2$. Futhermore, the isotopy class of an eyeglass $\eta$ is denoted by $[\eta]$.
\end{defn}

	It is not hard to see that the eyeglass twist $T_{\eta}$ depends only on the isotopy class of $\eta$. In analogy to the case for \emph{Dehn twists}, we have the following lemma.
\begin{lem}\label{conju}
Given any $\varphi \in \mathcal{G}(N, \Sigma)$, we have $T_{\varphi(\eta)}=\varphi \cdot T_{\eta} \cdot \varphi^{-1}$.

\end{lem}

When a lens of an eyeglass $\eta$ is decomposed into two disks, the corresponding eyeglass twist $T_{\eta}$ can be expressed as the composition of two new eyeglass twists. We write it as the following lemma. 
\begin{lem}\cite[Figure 8]{freedman_powell_2018}\label{compo}
	Let $\eta=(a,b ,\lambda)$ be an eyeglass in $N$. If $b$ is the band sum of two disks $b_1$ and $b_2$ along arc $\iota$ such that $\beta_i=\partial b_i$ is disjoint from both the bridge $\lambda$ and the disk $a$, then for $i=1,2$, we choose an proper arc $\lambda_{i}$ which connects $\alpha$ and $\beta_{i}$ and is disjoint from $\iota$, in the planar surface bounded by  $\alpha,\beta_1,\beta_2$ and $\gamma$.  Then we obtain two new eyeglasses, such as $\eta_1\left(=\left(a,b_{1},\lambda_{1}\right)\right)$ and $\eta_2\left(=\left(a,b_{2},\lambda_{2}\right)\right)$. Moreover, $T_{\eta}=T_{\eta_1} \cdot T_{\eta_2}.$ 
\end{lem}
\begin{figure}[h]
\includegraphics[width=0.5\textwidth]{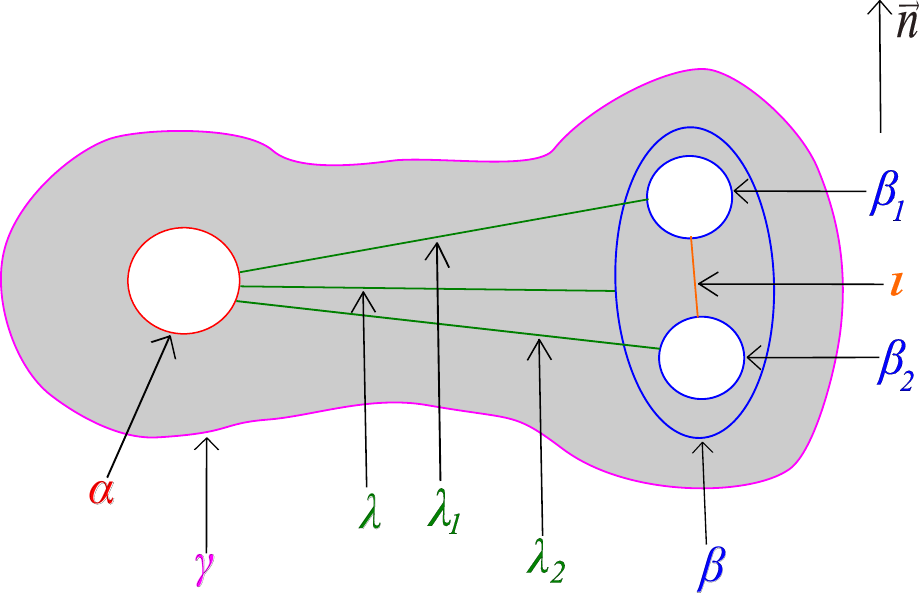}
\caption{composition of eyeglass twists}\label{com}
\end{figure}
\begin{proof}
	See Figure \ref{com} for the proof.
\end{proof}

\begin{defn}
	Suppose $\eta$ is an eyeglass in $N$, and $S$ a reducing sphere for $\Sigma$. We say $S$ separates $\eta$ if these two lenses of $\eta$ are disjoint from $S$ and lie in different components of $N\setminus S$.
\end{defn}
\begin{defn}
	Suppose $S$ is a reducing sphere for $\Sigma$, $\mu_s=S\cap\Sigma$, and $\eta$(with $\partial\eta=\alpha\cup\beta\cup\lambda)$) an eyeglass in $N$.  If $S$ is disjoint from the lenses of $\eta$,  a binary function $I(\eta,S)=\tilde{I}(\lambda,\mu_{s})$, where $\tilde{I}(\cdot,\cdot)$ is the geometric intersection number up to isotopies that leave $\alpha\text{ and }\beta$ invariant.
\end{defn}
%for the surface $\Sigma\setminus \left(\alpha\cup \beta\right)$
%We define a binary function $\mathcal{I} (\eta,S)=I(\mathfrak{su}(\eta),\mu_{s})$ to measure the complexity of intersection between $S$ and $\eta$, where $I(\cdot,\cdot)$ denotes the geometric intersection number for surface $\Sigma$.
%When $S$ is disjoint from the lenses of $\eta$, we have the inequation $\mathcal{I} (\eta,S)\leq2I(\eta,S)$.

\begin{defn}
	For any separating reducing sphere $S$ for  $\Sigma$, we associate it with a subgroup  $\mathcal{E}_{k}(S)$ of $\mathcal{G}(N,\Sigma)$ for each $k\in\mathbb{N}_{+} $ defined as follows:
	\begin{center}
		$\mathcal{E}_{k}(S)=\langle E_{k}(S)\rangle$,
	\end{center}
	where \begin{center}
	$E_{k}(S)=\left\{T_{\eta}\in \mathcal{G}(N,\Sigma): S\: \text{separates}\:  \eta \text{ and } I(\eta,S)\leq k\right\}$.	
	\end{center}
\end{defn}
By the definition, we have $E_k(S) \subset E_{k+1}(S)$. Then we have the ascending sequence:
$
\mathcal{E}_{1}(S) \leq \mathcal{E}_{2}(S) \leq \mathcal{E}_{3}(S) \leq\cdots.
$

Return to our setting that $N=V\cup_{\Sigma}W$ is a genus three \emph{Heegaard splitting} for a connected sum of two lens spaces and each $S_{i}$ is a separating reducing sphere for $\Sigma$. Then we have the following lemma.
\begin{lem}\label{reduce}
	For $i=1,2,3$ and $k\in\mathbb{N}_{+}$, we have $\mathcal{E}_{k+1}(S_i)\leq \mathcal{E}_{k}(S_i)$.
	\end{lem}
	\begin{proof}
It is sufficient to prove that $E_{k+1}\left(S_i\right) \subset \mathcal{E}_k\left(S_i\right)$. Consider an eyeglass twist $T_{\eta}$(suppose $\eta=(a,b ,\lambda)\text{ and }\partial\eta= \alpha\cup\beta\cup\lambda$) representing an element of $E_{k+1}\left(S_i\right)$.  We aim to show that $T_{\eta} \in E_k\left(S_i\right)$. Without loss of generality, we assume that $\beta$ lies in the genus 1 component of $\Sigma\setminus S_i$, and that the bridge $\lambda$ intersects $\mu_i(=S_i \cap \Sigma)$ at $k+1$ points. Let one of these points, say $p$,  be closest to $\alpha$, as depicted in Figure \ref{int}. The point $p$ divides $\lambda$ into two segments, $\lambda_1\text{ and }\lambda_2$, where $\lambda_1$ denotes the one that is disjoint from $\mu_i$; see Figure \ref{int}. 

Next, we choose an arc $\lambda_3$ in the pair of pants $\Sigma\setminus \left(S_{i}\cup\beta\right)$ that  connects the point $p$ and $\beta$, and whose interior does not intersect $\beta$; see Figure \ref{trans}. The choice of $\lambda_3$ is unique up to isotopy. Then we obtain a new eyeglass $\eta'=\left(a,b,\lambda_1\cup \lambda_3\right)$, with the corresponding two eyeglass twists $\varphi_1\text{ and }\varphi_2 \left(\text{i.e.}\: \varphi_1=T_{\eta'},\varphi_2=T_{\overline{\eta'}}\right)$. Note that $I(\eta',S_i)=1$, which implies that $\varphi_1,\varphi_2\in \mathcal{E}_{k}(S_i)$.
		
  Pushing the 1-handle $a\times I$ along  the path $\iota$, as illustrated by the green line in Figure \ref{trans}, produces an eyeglass twist. This twist is exactly $\varphi_{2}(=\tau^{-1}_{\alpha}\cdot\tau^{-1}_{\beta}\cdot\tau_{\gamma})$. After the excursion of the 1-handle $a\times I$ along path $\iota$, the intersection points $p$ and $p'$ are eliminated, as shown in Figure \ref{newpattern}.  This implies that $I\left(\varphi_2\left(\eta\right),S_i\right)<k+1$. By Lemma \ref{conju} , we then have $T_{\eta}=\varphi^{-1}_2 \cdot T_{\varphi_{2}(\eta)} \cdot \varphi_2$. 
%however, the other intersection points are left invariant and no new intersection points birth.
Since both $T_{\varphi_{2}(\eta)}$ and $\varphi_{2}$ belong to $\mathcal{E}_{k}(S_i)$, it follows that $T_{\eta}$ also belongs to $\mathcal{E}_{k}(S_i)$.
	\end{proof} 
	\begin{figure}[h]
		\includegraphics[width=0.6\textwidth]{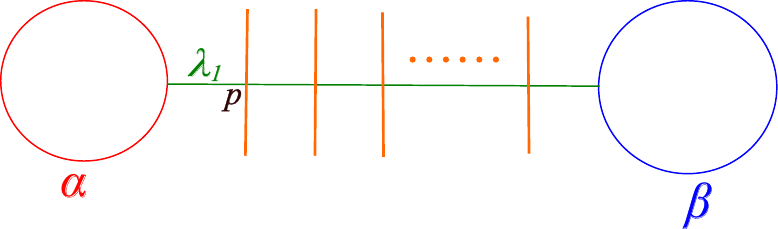}
		\caption{intersection pattern}\label{int}
	\end{figure}
	\begin{figure}[h]
		\centering
		\begin{subfigure}[b]{0.45\textwidth}
		\includegraphics[width=\textwidth]{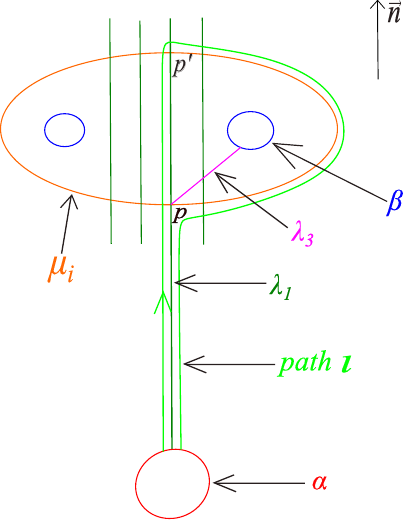}
		\caption{new eyeglass}\label{trans}
		\end{subfigure}\hfill
  \begin{subfigure}[b]{0.5\textwidth}
			\includegraphics[width=\textwidth]{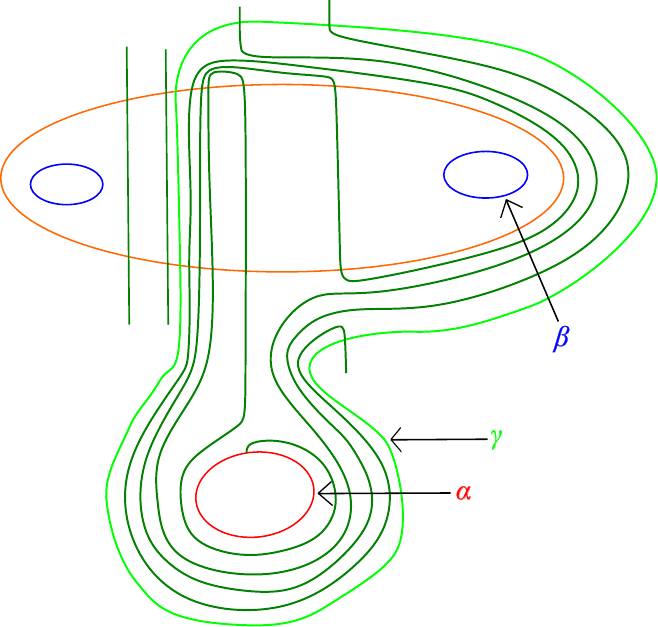}
			\caption{new intersection pattern}\label{newpattern}
		\end{subfigure}
		\caption{}
	\end{figure}
In summary, we have the equation\begin{center}
	$\mathcal{E}_{1}(S_i)=\mathcal{E}_{2}(S_i)=\mathcal{E}_{3}(S_i)=\cdots$
\end{center}for $i=1,2,3$.

In \cite{scharlemann_powells_2022}, Scharlemann defines a class of automorphisms of \emph{Heegaard splitting}, called bubble moves, which generalize  one of the five automorphisms proposed by Powell in \cite{powell_homeomorphisms_1980}. In this work, we extend the concept  to a more general setting.
\begin{defn}
		For a 3-manifold $N$ with \emph{Heegaard splitting} $N=A \cup_{\Sigma} B$, a bubble is a 3-submainfold $\mathcal{B}$ of $N$, whose boundary is a 2-sphere. If the boundary $\bdd\mb$ is a reducing sphere for $\Sigma$, we call $\mb$ a bubble for $\Sigma$. The genus of $\mathcal{B}\cap \Sigma$ is referred to as the \emph{genus} of $\mb$. In addition, a bubble is called trivial if it is a 3-ball, otherwise it is singular. 
\end{defn}
\begin{defn}\cite[Section 2]{scharlemann_powells_2022}\label{bubb}
Let $N=A \cup_{\Sigma} B$ be a \emph{Heegaard splitting}, and $\mathcal{B}$ a trivial bubble for $\Sigma$. A \emph{bubble move} is an isotopy of $\mathcal{B}$ along a closed path in $\Sigma\setminus int(\mb)$ that starts and ends at $\mathcal{B}$, returning $(\mathcal{B}, \mathcal{B} \cap \Sigma)$ to itself. See Figure \ref{bm}.
\end{defn}
	\begin{figure}[h]
	\centering
		\includegraphics[width=0.5\textwidth]{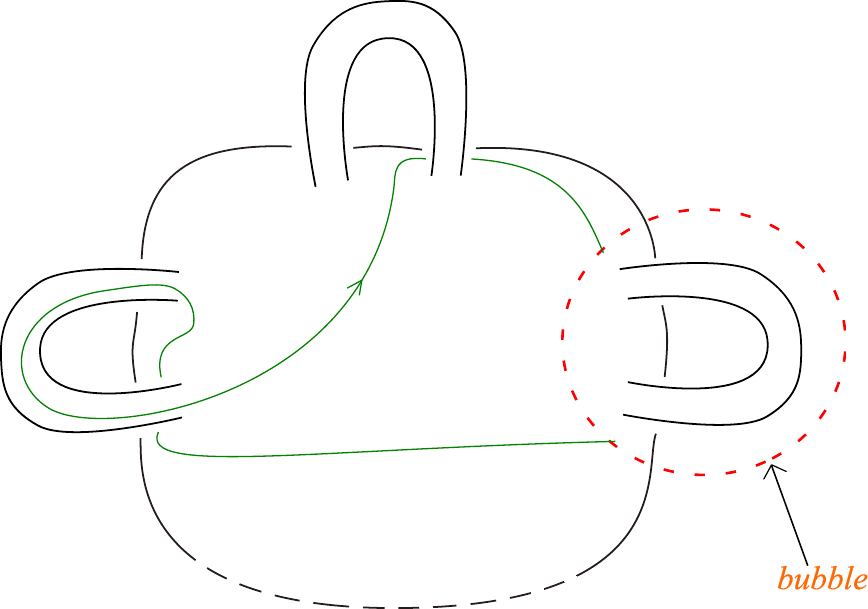}
		\caption{a bubble move}\label{bm}
	
\end{figure}

Given a genus 1 trivial bubble $\mathcal{B}$ for $\Sigma$, a bubble($\mb$) move represents an element of $\mathcal{G}(N, \Sigma)$. The subgroup of $\mathcal{G}(N, \Sigma)$ generated by all bubble($\mathcal{B}$) moves is denoted by $\mv_{\Sigma}(\mb)$. For a singular bubble $\mathcal{B}$, the definition of a bubble move does not directly apply. Therefore, we introduce a new class of automorphisms for singular bubbles.
\begin{defn}[\emph{visional bubble move}]\label{vb}
	Let $N=A \cup_{\Sigma} B$ be a \emph{Heegaard splitting} for a closed 3-manifold $N$, and $\mathcal{B}$ a bubble for $\Sigma$ with  $S$ as its boundary. The submanifold $N\setminus{int(\mb)}$ is also a bubble for $\Sigma$, which we refer to as the \emph{dual bubble} of $\mb$ and denote by $\mb'$. By capping the sphere boundary of $\mb'$ with a 3-ball, we obtain a new manifold $N(\mb)$.  The manifold $N(\mb)$ inherits a \emph{Heegaard splitting}, with its Heegaard surface $\Sigma'$ being the boundary sum of $\Sigma\setminus{int(\mb)}$ and a bordered torus(torus with an open disk removed), as illustrated in Figure \ref{hee}. Clearly, $N(\mb)\setminus int(\mb')$ is a trivial bubble in $N(\mb)$, bounded by the sphere $S$, denoted by $\mb^{3}$. As in the case of trivial bubbles,  a bubble($\mb^{3}$) move induces a diffeomorphism $h:N(\mb)\rightarrow N(\mb)$ so that $h|_{\mb^{3}}=\mathrm{id}$. We then glue the two diffeomorphisms $h|_{\mb'}:\mb'\rightarrow \mb'$ and $\mathrm{id}:\mb\rightarrow\mb$ along the sphere $S$ to obtain a diffeomorphism $\tilde{h}:(N,\Sigma)\rightarrow (N,\Sigma)$, which we refer to as a \emph{visional bubble($\mb$) move}. Then all visional bubble($\mb$) moves generate a subgroup of $\mg(N,\Sigma)$, denoted by $\mv_{\Sigma}(\mb)$.
\end{defn}
\begin{figure}[ht]
	\includegraphics[width=0.6\textwidth]{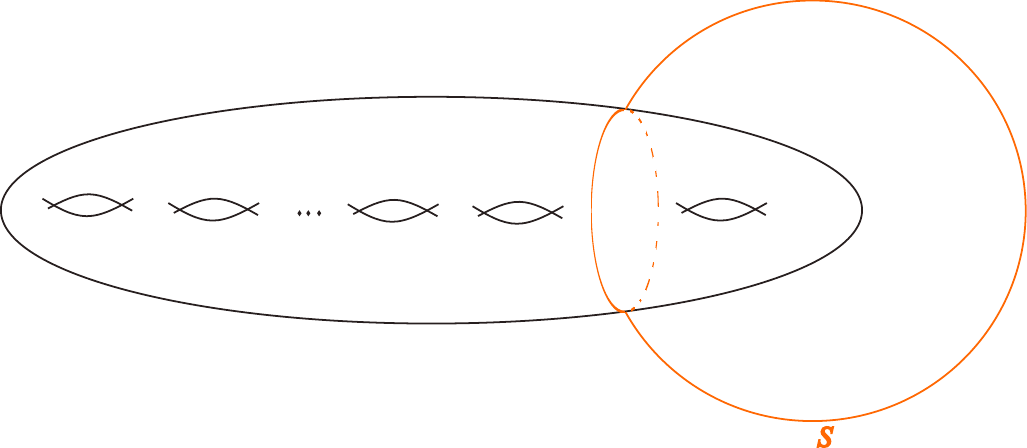}
	\caption{Heegaard surface $\Sigma'$}\label{hee}
\end{figure}
%\begin{rmk}
%	When $\mb$ is a trivial bubble, the visional bubble move is same as the bubble move.
%\end{rmk}
%\begin{defn}\cite[Section 2]{scharlemann_powells_2022}\label{flip}
%	Let $N=A \cup_\Sigma B$ be a \emph{Heegaard splitting}, and $\mathcal{B}$ a genus 1 bubble(maybe singular) for $\Sigma$. A \emph{flip} is the diffeomorphism $(\mathcal{B}, \mathcal{B} \cap \Sigma) \rightarrow(\mathcal{B}, \mathcal{B} \cap \Sigma)$ that reverses the orientation of both the meridian and longitude of the summand, as shown in Figure \ref{fl}.
%\end{defn}

In this setting, each sphere $S_i$ bounds a genus 1 bubble,  denoted by $\mb_{i}$. So we obtain a subgroup of $\mathcal{G}(N, \Sigma)$ for each $i$, say $\mv_{\Sigma}(\mb_{i})$, just defined as above. Additionally, we define a subgroup of $\mathcal{G}(N, \Sigma)$ for each $i$ by
$$
\mathcal{H}_i=\left\{\varphi \in \mathcal{G}(N, \Sigma): \varphi\left(\bar{\mu}_i\right)=\bar{\mu}_i\right\},
$$
 where $\mathcal{H}_i$ is the stabilizer of the isotopy class of the curve $\mu_i$. Furthermore, let $\mathcal{H}$ be the subgroup of $\mathcal{G}(N,\Sigma)$ generated by $\mathcal{H}_1,\mh_2,\text{ and }\mh_3$, i.e., $\mathcal{H}=\langle\mathcal{H}_1,\mathcal{H}_2,\mathcal{H}_3\rangle$. According to the definition of the visional bubble moves, we have $\mv_{\Sigma}(\mb_{i})\leqslant\mh_i$.

\section{Stabilizers of reducing 2-spheres}\label{sec4}
 To prove the main theorem,  we firstly show that $\mathcal{G}(N,\Sigma)=\mathcal{H}$ in Theorem \ref{th}. Then we will prove that $\mathcal{H}$ is finitely generated in the next section. As a corollary, the reducing sphere complex $\mr$ is connected. Before proving Theorem \ref{th}, we introduce two Lemmas \ref{sepe} and \ref{same}.
\begin{lem}\label{sepe}
	\quad $\mathcal{E}_1(S_i)< \mathcal{H}\quad$ for $i=1,2,3$.
\end{lem}
\begin{proof}
	It is sufficient to prove that all generators of $\mathcal{E}_1\left(S_i\right)$ belong to the subgroup $\mathcal{H}$, i.e., $E_1(S_i)\subseteq \mathcal{H}$, for $i=1,2,3$. Without loss of generality, we assume that $i=1$. For any element $T_{\eta}\in E_1(S_1)$, $\eta=\left(a,b,\lambda\right)$ is an eyeglass such that $S_1$ separates $\eta$ and the bridge $\lambda$ intersects $S_1$ transversely at one point. In this case, both $\alpha=\partial a$ and $\beta=\partial b$ are disjoint from 
$S_1\cap \Sigma$.  We assume the following conditions:

1. $a$(resp. $b$) is a disk in $V$(resp. $W$). 

2. $\alpha\text{ and }\beta$ lie in the genus 2 and the genus 1 component of $\Sigma\setminus S_1$ respectively.

The other case that $\alpha $ and $\beta$ lies in the same genus 2 component of $\Sigma\setminus S_1$  will be discussed in Lemma 4.2. Now there are three cases depending on
the intersection number between $\alpha$ and $\mu_3=S_3 \cap \Sigma$.
	
	Case 1. $\alpha \cap \mu_3=\varnothing$ and $\alpha$ lies in the genus -1 component of $\Sigma \setminus S_3$. 
 
 In this case, we apply a visional $\mb_{1}$ move  to reduce the intersection number $I(\eta,S_2)$. To be precise,  there exists a visional bubble  move $\psi_1 \in \mathcal{H}_1$ so that $$I\left(\psi_1\left(\eta\right),S_2\right)=0,\;\psi_1\left(\alpha\right)=\alpha\text{ and } \psi_1\left(\beta\right)=\beta.$$ The following Figure \ref{reduction} illustrates how a visional bubble move reduces the intersection.
 	\begin{figure}[ht]
		\centering
		\begin{subfigure}[b]{0.45\textwidth}
			\includegraphics[width=\textwidth]{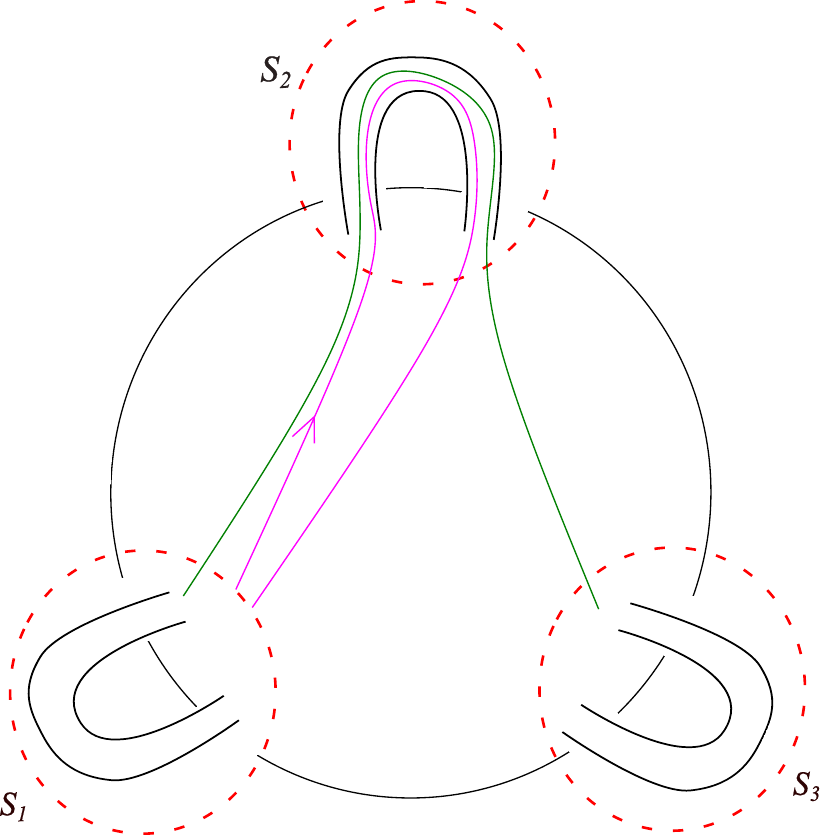}
			\caption{a visional $\mb_1$ move along the pink path}
		\end{subfigure}
		\hfill
		\begin{subfigure}[b]{0.45\textwidth}
		\includegraphics[width=\textwidth]{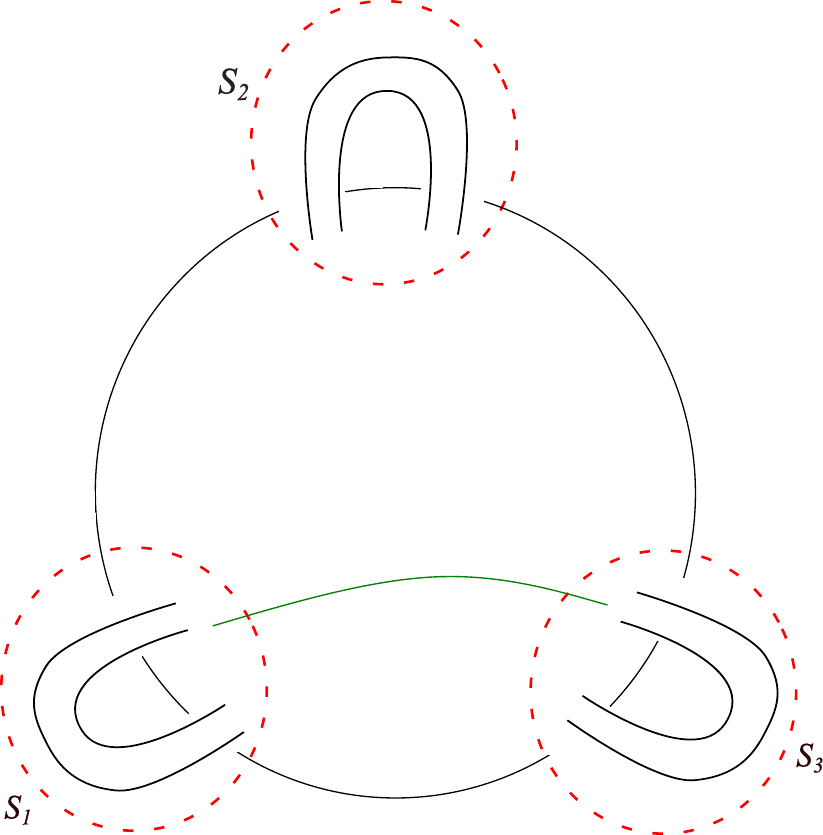}
		\caption{the reduction}
		\end{subfigure}
		\caption{}\label{reduction}
	\end{figure}
Note that $\psi(\eta)$ is disjoint from $S_2$. It means that $T_{\psi_{1}\left(\eta\right)}\in\mh_2$. Then we have
	$$
	T_{\eta}=\psi_{1}^{-1} \cdot T_{\psi_{1}\left(\eta\right)}\cdot\psi_{1} \in\mathcal{H}.
	$$
	
	Case 2. $\alpha \cap \mu_3=\varnothing$ and $\alpha$ lies in the genus 2 component of $\Sigma \setminus S_3$. Similarly, we can find a visional $\mb_{1}$ move $\psi_2 \in \mathcal{H}_1$ so that $$I\left(\psi_2\left(\eta\right),S_3\right)=0,\;\psi_2\left(\alpha\right)=\alpha\text{ and }\psi_2\left(\beta\right)=\beta.$$ 
	It follows that $T_{\psi_{2}(\eta)}\in\mh_3$ and $T_{\eta}=\psi_{2}^{-1} \cdot T_{\psi_{2}(\eta)}\cdot\psi_{2}\in\mathcal{H}$.
	
	Case 3. $\alpha \cap\mu_3 \neq \varnothing$. The proof is by induction on the geometric intersection member $I\left(\alpha, \mu_3\right)$. Since $\mu_3$ is separating, $I\left(\alpha, \mu_3\right)$ is an even number. Assume the statement is true for $I\left(\alpha, \mu_3\right) \leq 2n$. Now we consider the case that $I\left(\alpha, \mu_3\right)=2n+2$. We first isotope $\eta$ so that all the endpoints of the bridge $\lambda$ to lie in the genus 2 component of $\Sigma\setminus S_{3}$, while keeping  the curves $\alpha\text{ and }\beta$ invariant during the isotopy. Next, we apply  a visional bubble $\mb_{1}$ move $\bar{h} \in \mathcal{H}_1$  along $\lambda$ such  that $I\left(h(\alpha), \mu_3\right) \leqslant I\left(\alpha, \mu_3\right)$ and $h(\lambda)$ is disjoint from $\mu_3$. 
	
    Let $D$(resp. $D'$) be the disk bounded by $h(\alpha)$(resp. $\mu_3$) in $V$. Without loss of generality, we assume that $\mid D\cap D'\mid$ is minimal. Then there is an outermost disk $D''$ in $D'$ relative to $D$. Doing a compression on $D$ along $D''$ results in two essential disks $D_{1}$ and $D_{2}$, with boundaries $\alpha_{1}$ and $\alpha_2$. Because  $\alpha_1\cup\alpha_2$ intersects $h(\lambda)$ at most once, we now have two subcases to consider.

Subcase 3.1. $\left(\alpha_1\cup\alpha_2\right)\cap h(\lambda)=\varnothing$. By Lemma \ref{compo}, $T_{h(\eta)}$ is a composition of two eyeglass twists whose bases have less intersection with $\mu_3$. Since the intersection number is smaller,  the induction hypothesis applies. Therefore, by the inductive assumption, we have $T_{h(\eta)}\in\mathcal{H}$. 

Subcase 3.2.  $|\left(\alpha_1\cup\alpha_2\right)\cap h(\lambda)|=1$. Without loss of generality,
we assume that $\alpha_1$ intersects $h(\lambda)$ at one point. We now construct two new eyeglasses $\eta_1\left(=\left(D_1,b,\lambda_{1}\right)\right)\text{ and }\eta_2\left(=\left(D_{2},b,\lambda_{2}\right)\right)$, as depicted in Figure \ref{neweyeglasses}.  It is not hard to see that $T_{h(\eta)}$ can be expressed as a composition of $T_{\eta_1}$ and $T_{\eta_2}$.  It follows that $T_{h(\eta)}\in\mh$.
%Since $h(\lambda)$ incident to $h(a)$ is disjoint from $\mu_3$,   $\alpha_{1}\cup \alpha_{2}$ intersects $h(\lambda)$ in one point.Without loss of generality, we assume that $\alpha_{1}$ intersects $h(\lambda)$ at one point. Then we have  two new eyeglasses $\eta_1=\left(D_1,b,\lambda_{1}'\right))\text{ and }\eta_2=\left(D_2,b,\lambda_{2}'\right)$, depicted as in Figure \ref{neweyeglasses}.  It is not hard to see that $T_{h(\eta)}$ is generated by $T_{\eta_1}$ and $T_{\eta_2}$. By construction, both $\mid \alpha_1 \cap \mu_3 \mid $ and $\mid \alpha_2 \cap \mu_3 \mid $ is strictly less than $\mid \alpha\cap \mu_3\mid$. Hence $\mid \alpha_i \cap \mu_3 \mid \leq 2n$ for $i=1,2$. By induction, $T_{\eta_i} \in \mathcal H$, for $i=1,2$. It follows that $T_{h(\eta)}\in\mh$. So $T_{\eta}\in \mathcal H$.

The proof for the other cases are similar. So we omit it.
\end{proof}
\begin{figure}[ht]
	\includegraphics[width=0.6\textwidth]{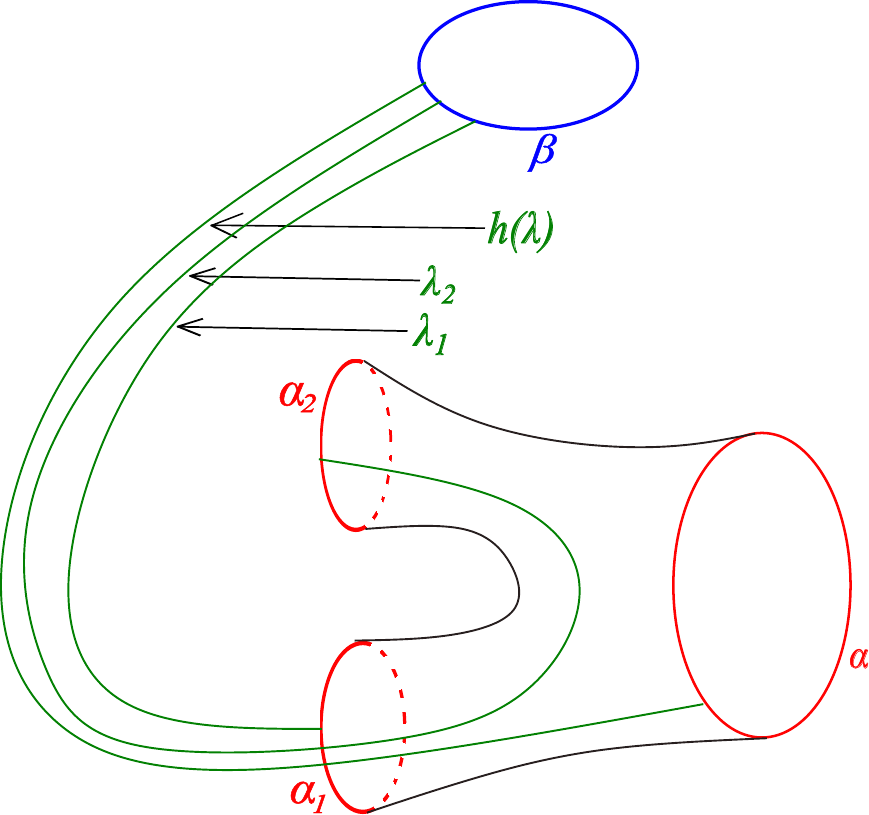}
	\caption{new eyeglasses}\label{neweyeglasses}
\end{figure}
\begin{lem}\label{same}
	For any eyeglass $\eta=(a,b,\lambda)\left(\text{with }\partial\eta=\left(\alpha,\beta,\lambda\right)\right)$ satisfying that both $\alpha\text{ and }\beta$ lie in the genus 2 component of $\Sigma\setminus S_i$$\left(i=1\:\text{or}\:2\right)$, we have $T_{\eta}\in\mathcal{H}$.
\end{lem}
\begin{proof}
Without loss of generality, we assume that both $\alpha$ and $\beta$  lie in the genus 2 component of $\Sigma\setminus S_1$. The other case is similar.  So we omit it.  If $\lambda$ is disjoint from $\mu_1=S_1\cap \Sigma$, then $T_{\eta}$ fixes 
$\mu_{1}$ and hence it lies in $\mathcal{H}$. So we assume that neither $\alpha$ nor $\beta$ is isotopic to $\mu_1$, and  that $\lambda \cap \mu_1\neq \varnothing$.  After doing three compressions on $\Sigma$ along $a$, $b$, and the disk $D$ bounded by $\mu_1$, we obtain a  resulting sphere $S$. 
\begin{clm}
    The sphere $S$ contains a scar of $D$.
\end{clm}
\begin{proof}
    Suppose the conclusion is false.  In this case, one of $\alpha$ and $\beta$ is separating in $\Sigma$, says $\alpha$. Then $\beta$ lies in the genus 1 component of $\Sigma\setminus \alpha$. Since $\alpha$ is disjoint from  $\mu_1$, it means that $\lambda$ is disjoint from $\mu_1$, a contradiction.
\end{proof}
Hence $S$ contains $D$ and the scars of $a$ and $b$. Then we choose two disjoint simple closed curves $\ell_1,\ell_2 \subset S$ such that  $\ell_1$ cuts off a disk containing only the scars of $a$, while $\ell_2$ cuts off a disk containing only the scars of $b$. It is not hard to see that both $\ell_1\text{ and }\ell_2$ are reducing curves. In other words,	there are two disjoint reducing sphere $S_{\ell_1}\text{ and }S_{\ell_2}$ so that $S_{\ell_i} \cap \Sigma=\ell_i$. By definition, these three spheres-$S_{\ell_1}, S_{\ell_2,}\text{ and }S_1$-constitute a complete sphere triplet, denoted by $\T'$. Note that $S_1$ is a common 2-sphere between $\T\left(=\left(S_1,S_2,S_3\right)\right)\text{ and }\T'$. By Lemma \ref{stab}, there exists an element $\phi \in \mh_{i}$ so that $\phi\left(\T'\right)=\T$. To prove $T_{\eta}\in 
 \mathcal {H}$, it is sufficient to prove $T_{\phi(\eta)} \in\mathcal{H}$.

 If $\alpha$ is isotopic to $\ell_1$, then $\phi(\alpha)$ is isotopic to one of $\mu_{1},\mu_2,$ and $\mu_3$. So $T_{\phi(\eta)}\in H$. Otherwise, $S_{\ell_1}$ separates $\eta$. Then $\phi(S_{\ell_1})$ separates $\phi(\eta)$. By Lemma \ref{reduce} and the argument in Lemma \ref{sepe}, we have$$T_{\phi(\eta)} \in E_{I\left(\phi\left(\eta\right),\phi(S_{\ell_1})\right)}\left(\phi(S_{\ell_1})\right)\subseteq\mathcal{E}_{I\left(\phi\left(\eta\right),\phi(S_{\ell_1})\right)}\left(\phi(S_{\ell_1})\right)=\mathcal{E}_1\left(\phi(S_{\ell_1})\right)<\mathcal{H}.$$
\end{proof}

Recently, the classical ``Haken's lemma" has been strengthened into the ``strong Haken's lemma" by several authors in various ways\cite{scharlemann_strong_2024,hensel_strong_2021,taylor_strong_2024}, which says that a sphere set in a 3-manifold can be isotoped to be aligned with the given Heegaard surface (see the following definition for ``aligned").

\begin{defn}\cite[Section 1]{freedman_uniqueness_2024}
A sphere set $E\subset N$ is a compact properly embedded surface in $N$ so that each component of $E$ is a sphere. Then, a Heegaard surface $\Sigma$ and a sphere set $E$ in $N\left(=A\cup_\Sigma B\right)$ are {\em aligned} if they are transverse, and each component of $E$ intersects $\Sigma$ in at most one circle. In addition, each disk component of $E\setminus\Sigma$ is essential in either $A$ or $B$.  
\end{defn}
 Freedman-Scharlemann\cite{freedman_uniqueness_2024} prove that not only can we align a given sphere set with a Heegaard surface, but also any two alignments can be related by a sequence of \emph{bubble moves}\cite[Definition 1.2]{freedman_uniqueness_2024} and \emph{eyeglass twists}\cite[Definition 1.3]{freedman_uniqueness_2024}.
\begin{note}
	The same word ``\emph{bubble move}" used here does not refer to an automorphism of $(N,\Sigma)$ as in definition \ref{bubb}. It is just an isotopy of a sphere, which pushes a sphere across a trivial bubble; see details in \cite[Definition 1.2]{freedman_uniqueness_2024}. Throughout the following of the article, we will not distinguish them in their notations when it does not cause confusion.
\end{note}

\begin{defn} \cite[Definition 1.4]{freedman_uniqueness_2024}
	Let $E_0$ and $E_1$ be two sphere sets aligned with $\Sigma$ in $N$.  $E_0, E_1$ are equivalent if there is an isotopy $H: N\times 
 I\rightarrow N$ so that $H_{s}:\Sigma\times \{s\}\rightarrow \Sigma$ for $0\leq s \leq 1$ and $H_{1}$ maps $E_0$ to $E_1$.
\end{defn}

\begin{thm}\cite[Theorem 1.6]{freedman_uniqueness_2024}\label{align}
	
	If $E_0, E_1$ are sphere sets that are properly isotopic in $N$ and are both aligned with $\Sigma$, then up to equivalence, $ E_1$ can be obtained from $E_0$ by a sequence of bubble moves and eyeglass twists.
\end{thm}

Note that there is a bijection between the isotopy classes of reducing curves and the isotopy classes of reducing spheres. We commonly identify a reducing sphere with its corresponding reducing curve. 

\begin{thm} \label{th} 
    $\mathcal{G}(N,\Sigma)=\mathcal{H}$
\end{thm}
\begin{proof}
    We divide the proof into two cases: (1) $N_{1}\neq N_{2}$; (2) $N_{1}=N_{2}$. 
\subsection*{case 1.} 
Let $\mr^{0}$ be the $0$-skeleton of $\mr$. As both $\mg\left(N,\Sigma\right)$ and $\mh$ naturally act on  $\mr^{0}$, we denote the corresponding orbit containing the isotopy class $\bar{\mu}_{i}$ by $\mathscr{O}_i\text{ and }\mathscr{O}_{i}'$ respectively. Since we have assumed that $N_1\neq N_2$, we know that $\mathscr{O}_1\cap\mathscr{O}_2=\varnothing$. To prove the theorem, it suffices to show that $\mathscr{O}_1=\mathscr{O}'_1$. 

Given any reducing sphere $S$ for $\Sigma$ with the intersection curve $\mu\left(=S\cap \Sigma\right)$ representing an isotopy class $\bar{\mu}$ of $\mathscr{O}_1$, we will prove that $\bar{\mu}\in\mathscr{O}'_1$; if it does, we immediately obtain the desired result $\mathscr{O}_1=\mathscr{O}'_1$.
	
	By an innermost argument, these two essential spheres $S$ and $S_1$ are isotopic. Moreover, both $S$ and $S_1$ are aligned with $\Sigma$. By Theorem \ref{align}, $S$ is related to $S_1$ by a sequence of bubble moves and eyeglass twists. It means that there is a sequence of reducing curves $\lambda_{i}$ in $\Sigma$ so that $\lambda_{i+1}$ can be obtained from $\lambda_i$ by a bubble move or an eyeglass twist
	$$
	\mu_1=\Sigma \cap S_{1}=\lambda_1, \lambda_2, \lambda_3, \cdots \lambda_n=\Sigma \cap S=\mu.
	$$
	
	We first prove by induction that $\bar{\lambda}_i \in \mathscr{O}'_1\cup\mathscr{O}'_2$. The base case $\bar{\lambda}_1\in\mathscr{O}_1^{\prime}\cup\mathscr{O}_2^{\prime}$ clearly holds. Now we assume that $\bar{\lambda}_i \in\mathscr{O}_1^{\prime} \cup\mathscr{O}_2^{\prime}$. Then we prove that $\bar{\lambda}_{i+1} \in\mathscr{O}_1^{\prime} \cup\mathscr{O}_2^{\prime}$. Without loss of generality, we assume $\bar{\lambda}_i \in\mathscr{O}_1^{\prime}$. It means that there is an element $\bar{g}\in\mathcal{H}$ so that $\bar{g}\left(\bar{\lambda}_i\right)=\bar{\mu}_1$. Let $g \in\operatorname{Diff}^{+}(N, \Sigma)$ be a representative of $\bar{g}$ so that $g\left(\lambda_i\right)=\mu_1$. Then there are two cases as follows.

	Case 1. $\lambda_{i+1}$ is obtained from $\lambda_i$ by a bubble move. It implies that $g\left(\lambda_{i+1}\right)$ can be obtained from $g\left(\lambda_i\right)$ by a bubble move. In this buddle move, the trivial bubble has a sphere boundary $S_\ell$, where $S_\ell \cap \Sigma=\ell$. Then these three curves $g\left(\lambda_{i+1}\right), \ell$, and $\mu_1\left(=g\left(\lambda_i\right)\right)$ cobound a pair of pants illustrated as in Figure \ref{newpants}. Denote $S_{g\left(\lambda_{i+1}\right)}$ the reducing 2-sphere  intersecting transversely with $\Sigma$ in $g\left(\lambda_{i+1}\right)$. Note that  $S_1\cup S_\ell\cup S_{g\left(\lambda_{i+1}\right)}$ divides $N$ into four submanifolds of the same diffeomorphism type as those divided by $S_1\cup S_3\cup S_2$. It follows that there is a diffeomorphism $h\in\operatorname{Diff}^{+}\left(N,\Sigma\right)$  so that  $h\left(S_1,S_\ell,S_{g\left(\lambda_{i+1}\right)}\right)=\left(S_1, S_3, S_2\right)$. Thus, we have $h \cdot g\left(\lambda_{i+1}\right)=\mu_2$. Furthermore, $\bar{h}\cdot \bar{g}\left(\bar{\lambda}_{i+1}\right)=\bar{\mu}_2$, where $\bar{g}\in\mathcal{H}$ and $\bar{h} \in \mathcal{H}_1$. So $\bar{\lambda}_{i+1} \in\mathscr{O}_2^{\prime}$.
	\begin{figure}[ht]
		\includegraphics[width=0.6\textwidth]{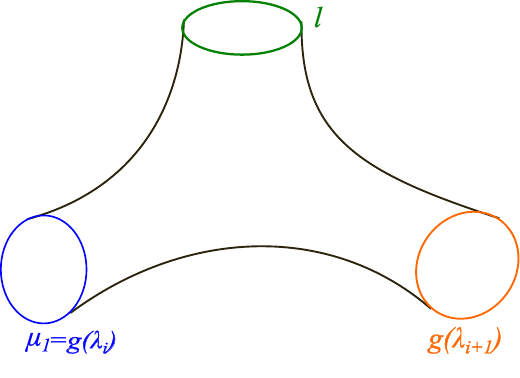}
		\caption{a pair of pants bounded by $g\left(\lambda_{i+1}\right), \ell$, and $\mu_1$}\label{newpants}
	\end{figure}
	
	Case 2. $\lambda_{i+1}$ is obtained from $\lambda_i$ by an eyeglass twist. It means that $g\left(\lambda_{i+1}\right)$ can be obtained from $g\left(\lambda_i\right)\left(= \mu_1\right)$ by an eyeglass twist $T_\eta$. There are two subcases as follows.
	
	Subcase 2.1. $S_1$ separates $\eta$. By Lemma \ref{reduce} and Lemma \ref{sepe}, we have $T_\eta \in\mathcal{H}$. As $T_{\eta}^{-1}\cdot g\left(\lambda_{i+1}\right)=\mu_1$, $\bar{\lambda}_{i+1} \in\mathscr{O}_1^{\prime}$.

	Subcase 2.2. $S_1$ does not separate $\eta$. Then these two lenses of $\eta$ both lie in the component of $N\setminus S_1$ which contains the genus 2 component of $\Sigma\setminus S_1$. By Lemma \ref{same}, $T_{\eta} \in\mathcal{H}$. So we have $\bar{\lambda}_{i+1} \in\mathscr{O}'_1$.
	
	By the above argument, we have completed the proof for the statement that $\bar{\lambda}_i \in \mathscr{O}_1^{\prime}\cup\mathscr{O}_2^{\prime}$ for all $i \leqslant n$. In particular, $\bar{\mu}=\bar{\lambda}_n \in\mathscr{O}_1^{\prime} \cup\mathscr{O}_2^{\prime}$. However, $\bar{\mu}\in\mathscr{O}_1, \mathscr{O}_1 \cap\mathscr{O}_2=\varnothing,\text{ and } \mathscr{O}_i^{\prime} \subset\mathscr{O}_i$; it implies that $\bar{\mu}\in\mathscr{O}_1^{\prime}$. This completes the proof of the case 4.1.
\subsection*{case 2.}
	By the same argument as above, we can construct a diffeomorphism $f\in\operatorname{Diff}^{+}\left(N, \Sigma\right)$ so that $f\left(S_1,S_2,S_3\right)=\left(S_2,S_1,S_3\right)$. It means that $\mathscr{O}_1^{\prime}=\mathscr{O}_2^{\prime}$ . And we can apply the above method  to show that $\mathscr{O}_1^{\prime}=\mathscr{O}_1$  with only slight modifications. Overall, $\mg(N, \Sigma)=\mh$.
\end{proof}

  We use the symbol $\mr$ to represent the reducing sphere complex for the \emph{Heegaard splitting} $N=V\cup_\Sigma W$, which is a subcomplex of the curve complex spanned by those curves that bound disks in both two handlebodies.  Now we use the above results to prove the connectedness of the reducing sphere complexes.
%When clarity is not compromised, we just denote it by $\mr$ without mentioning the Heegaard surface $\Sigma$.
\begin{proof}[Proof of Corollary \ref{reducing}]
The reducing 2-spheres $S_1,S_2\text{ and }S_3$ are contained in a same component of $\mr$, says $\mr'$. To prove the connectedness of $\mr$, it suffices to show that $\mr=\mr'$.
	
	Given any reducing sphere $S\in \mr$, there are two other reducing spheres $S'\text{ and }S''$ in $\mr$ such that the sphere triplet $\left(S,S',S''\right)$ is complete. Denote $\T'=\left(S,S',S''\right)$. By Lemma \ref{transitive}, there is an element $\nu\in\mg(N,\Sigma)$ so that $\nu\left(\T\right)$ is congruent to $\T'$. By Theorem \ref{th}, we have$$\nu=\theta_{n}\cdot\theta_{n-1}\cdots\theta_{2}\cdot\theta_{1},$$where $\theta_{i}\in \mh_1\cup\mh_2\cup\mh_3.$
	
	Let $\nu_i=\theta_{i}\cdot\theta_{i-1}\cdots\theta_{2}\cdot\theta_{1}\left(0\leqslant i\leqslant n\right)$ and $\T_i=\nu_{i}(\T)$. Then we obtain a sequence of triplets:$$\T=\T_0,\T_1,\T_2\cdots\cdots\T_{n-1},\T_{n}=\T'.$$
	We identify a complete sphere triplet with a 2-simplex of $\mr$. Then we prove by induction that $\T_i\subset\mr'$. The base case $\T_0\subset\mr'$ clearly holds. Now we assume that $\T_k\subset\mr'$. Without loss of generality, we assume that $\theta_{k+1}\in\mh_1$. By induction assumption, $\T_{k}$ and $S_1$ are contained in the same component $\mr'$. So $\theta_{k+1}\left(\T_{k}\right)$ and $\theta_{k+1}(S_1)$ are also in the same component of $\mr$. Since $\theta_{k+1}(S_1)=S_1\in\mr'$, $\T_{k+1}=\theta_{k+1}\left(\T_{k}\right)\subset\mr'$. Therefore $S\in\T'=\T_{n}\subset\mr'$. This completes the proof.
\end{proof}
\section{Finite generation}
\label{sec5}
We prove that each $\mh_i$ is finitely generated, and so is $\mathcal{H}$. Before starting our proof, we introduce a definition as follows.
\begin{defn}
	Let $N=A \cup_{\Sigma} B$ be a \emph{Heegaard splitting} and $D$ a collection of finitely many marked disks in $\Sigma$. In this case, we define the \emph{Goeritz group} $\mg\left(N,\Sigma,D\right)$ to be the group of diffeomorphisms of $N$ that preserve the \emph{Heegaard splitting} and the disk set $D$, modulo isotopies that leave $\Sigma$ and $D$ invariant.
\end{defn}
 Suppose $\mb$ is a bubble for $\Sigma$, with the boundary sphere $S=\partial \mb$ intersecting $\Sigma$ in an essential simple closed curve $\mu$. Let $D_{a}$ be the disk $S\cap A$ and $\Sigma_{\mb}$ the bordered surface  $\Sigma\setminus int(\mb)$. Denote $\Sigma(\mb)$ the closed surface $\Sigma_{\mb}\cup D_{a}$. Since $S$ is separating in $N$, we cap the sphere boundary $S$ of the dual bubble $\mb'\left(=N\setminus int\left(\mb\right)\right)$ with a 3-ball to obtain a new closed orientable 3-manifold $N(\mb)$. 
 It is not hard to see that $\Sigma(\mb)$ is a Heegaard surface for $N\left(\mb\right)$.  Placing an orientation for the curve $\mu$, we have the oriented curve $\vec{\mu}$ and its isotopy class  $[\vec{\mu}]$. Denote $G_{\vec{\mu}}$ the stabilizer of $[\vec{\mu}]$, i.e., $$G_{\vec{\mu}}\defeq\left\{\phi\in\mg\left(N,\Sigma\right): \phi\left([\vec{\mu}]\right)=[\vec{\mu}]\right\}.$$
%Let $\dd(N,\Sigma)$  be a subgroup of $\dd(N)$ defined by
 %$$\dd(N,\Sigma)\defeq\left\{f\in\dd(N): f(\Sigma)=\Sigma\text{ and }f\text{ preserves the orientation of }\Sigma\right\}.$$ the natural homomorphism $\rho_{1}: \dd(N,\Sigma)\rightarrow\mg\left(N,\Sigma\right)$ is an epimorphism. If an orientation-preserving diffeomorphism of $N$ preserves the \emph{Heegaard splitting} of $N$, it preserves the orientation of $\Sigma$. Hence, 
 
%Let $\mathcal{OC}$ be the set of isotopy classes of \emph{oriented essential simple closed curves} in $\Sigma$. The \emph{Goeritz group} $\mg\left(N,\Sigma\right)$ naturally acts on $\mathcal{OC}$. 

Now we define a subgroup of $\dd\left(N,\Sigma\right)$ as follows:
$$\dd\left(N,\Sigma,S\right)\defeq\left\{f\in\dd\left(N,\Sigma\right): f|_{S}=\mathrm{id}\right\}.$$By the definition, if $f\in\dd\left(N,\Sigma,S\right)$, then $f(\mb'=\overline{N-\mb})=\mb'$. Then we associate it with an element of the \emph{Goeritz group} $\mg\left(N(\mb),\Sigma(\mb), D_{a}\right)$  as follows: since  $N(\mb)$ is the union of $B^{'}$ and a 3-ball, $f\mid_{\mb'}$ can be naturally extended into a diffeomorphism $\hat{f}: \left(N\left(\mb\right),\Sigma(\mb),D_{a}\right)\rightarrow\left(N\left(\mb\right),\Sigma(\mb),D_{a}\right)$. Note that all different extensions of $f$ are pairwise isotopic. Subsequently, we obtain a map $\rho_2: \dd\left(N,\Sigma,S\right)\rightarrow\mg\left(N\left(\mb\right),\Sigma\left(\mb\right),D_{a}\right)$. It is not hard to verify that $\rho_2$ is also an epimorphism. 

Restricting the natural homomorphism $\rho_1:\dd(N,\Sigma)\rightarrow\mg\left(N,\Sigma\right)$ to the subgroup $\dd\left(N,\Sigma,S\right)$ results in a restriction map, which we still denote by $\rho_1$. It is easy to see that $\rho_{1}\left(\dd\left(N,\Sigma,S\right)\right)=G_{\vec{\mu}}$. %and $\rho_{1}|_{\dd\left(N,\Sigma,S\right)}$ is an epimorphism.

Then we want to define a homomorphism $\rho: G_{\vec{\mu}}\rightarrow\mg\left(N\left(\mb\right),\Sigma\left(\mb\right),D_{a}\right)$ so that the following diagram commutes.
$$
\begin{tikzcd}
	\dd\left(N,\Sigma,S\right)\arrow[rd,"\rho_2"]  \arrow[r, "\rho_1"] & G_{\vec{\mu}} \arrow[d,"\rho"]\\
	& \mg\left(N\left(\mb\right),\Sigma\left(\mb\right),D_{a}\right)
\end{tikzcd}
$$
Note that if such $\rho$ exists, it is uniquely determined by the requirement. And its existence is guaranteed by the following lemma.
	\begin{lem}
		For any diffeomorphism $f \in\dd\left(N,\Sigma,S\right)$, if $\rho_1(f)=\mathrm{id}$, then $\rho_2(f)=\mathrm{id}$.
	\end{lem}
	\begin{proof}
		Assume $\vec{\alpha}$ is an oriented essential simple closed curve in $\Sigma_{\mb}$. As $\rho_{1}(f)=\mathrm{id}$, $f(\vec{\alpha})$ is isotopic to $\vec{\alpha}$ in $\Sigma$. By \cite[Lemma 3.16]{farb_primer_2012}, we know that $f(\vec{\alpha})$ is also isotopic to $\vec{\alpha}$ in $\Sigma_{\mb}$. In other words, $f$ preserves the isotopy classes of all oriented essential simple closed curves in $\Sigma_{\mb}$. Then we can prove by induction on genus that such a diffeomorphism is isotopic to a power of \emph{Dehn twist} along the boundary curve $\mu$. It means that $\rho_{2}(f)=\mathrm{id}$.
	\end{proof}
	It follows that such $\rho$ exists and is surjective, and the kernel of $\rho$ is denoted by $\mi(\rho)$. Then we have the following exact
 sequence
 \begin{equation}
     1\rightarrow\mi(\rho)\overset{i}{\rightarrow} G_{\vec{\mu}}\overset{\rho}{\rightarrow}\mg\left(N\left(\mb\right),\Sigma\left(\mb\right),D_{a}\right)\rightarrow 1.
 \end{equation}
	
	Similarly,  for the dual bubble $\mb'$, we also have a dual exact sequence
 \begin{equation}
     1\rightarrow\mi(\rho')\overset{i'}{\rightarrow} G_{\vec{\mu}}\overset{\rho'}{\rightarrow}\mg\left(N\left(\mb'\right),\Sigma\left(\mb'\right),D_{a}\right)\rightarrow 1.
 \end{equation}
%Combining the above two sequences, we obtain a new sequence:
%\begin{equation}
	%1\rightarrow\mi(\rho')\cap\mi(\rho)\overset{i''}{\rightarrow}G_{\vec{\mu}}\overset{\rho\times\rho'}{\longrightarrow}\mg\left(N\left(\mb\right),\Sigma\left(\mb\right),D_{a}\right)\times\mg\left(N\left(\mb'\right),\Sigma\left(\mb'\right),D_{a}\right)\rightarrow 1.
%\end{equation} 
%The exactness of sequence (3) is guaranteed by the following lemma.
\begin{lem}
    The composite homomorphism 	$\rho'\cdot i$ is an epimorphism.
\end{lem} 
\begin{proof}
By a similar argument for the dual bubble $\mb'$, we have the following commutative diagram 
$$
\begin{tikzcd}
	\dd\left(N,\Sigma,S\right)\arrow[rd,"\rho'_2"]  \arrow[r, "\rho_1"] & G_{\vec{\mu}} \arrow[d,"\rho'"]\\
	& \mg\left(N\left(\mb'\right),\Sigma\left(\mb'\right),D_{a}\right).
\end{tikzcd}
$$

Since $\rho_2$ is surjective, for any element $\phi\in \mg\left(N\left(\mb'\right), \Sigma\left(\mb'\right),D_{a}\right)$, we can find a diffeomorphism $f\in\dd\left(N,\Sigma,S\right)$ so that $\rho'_{2}(f)=\phi$. Then We  extend $f|_{\mb}$ by the identity to obtain a diffeomorphism $\hat{f}\in\dd\left(N,\Sigma,S\right)$. It is not hard to see that $\rho'\cdot \rho_{1} (\hat{f})=\rho'_{2}(\hat{f})=\rho'_{2}(f)=\phi$ and $\rho_{1}(\hat{f})\in\mi(\rho)$. Hence the lemma follows immediately.
\end{proof}
Subsequently, we have the following exact squence
\begin{equation}
     1\rightarrow\mi(\rho')\cap\mi(\rho)\overset{i'''}{\rightarrow} \mi(\rho)\overset{\rho'\cdot i}{\longrightarrow}\mg\left(N\left(\mb'\right),\Sigma\left(\mb'\right),D_{a}\right)\rightarrow 1.
\end{equation}
	Since the marked disk $D_{a}$ can be treated as a marked point in $\Sigma(\mb)$, we  apply the description for the kernel of \emph{the capping homomorphism}\cite[Proposition 3.19]{farb_primer_2012} to obtain that $\mi(\rho')\cap\mi(\rho)=\langle\tilde{\tau}_{\mu}\rangle$, where $\tilde{\tau}_{\mu}$ is the extension of \emph{Dehn twist} $\tau_{\mu}$ to the whole manifold $N$.
	\begin{lem}\label{ost}
	If  both $\mg\left(N\left(\mb\right),\Sigma\left(\mb\right)\right)$ and $\mg\left(N\left(\mb'\right),\Sigma\left(\mb'\right)\right)$ are finitely generated(or finitely presented), then $G_{\vec{\mu}}$ is finitely generated(or finitely presented) as well.
	\end{lem}
	\begin{proof}
		If the genus $g\left(\Sigma\left(\mb\right)\right)=1$, we have $\mg\left(N\left(\mb\right),\Sigma\left(\mb\right),D_{a}\right)=\mg\left(N\left(\mb\right),\Sigma\left(\mb\right)\right)$. Then there is nothing to prove. If $g\left(\Sigma\left(\mb\right)\right)\geqslant 2$, we apply the \emph{Birman exact sequence} for the pair $\left(\Sigma\left(\mb\right),D_{a}\right)$ to obtain the following commutative diagram,
$$
\begin{tikzcd}
  1 \arrow[r] & K \arrow[r, "Push"] \arrow[d, "i"] & \mg\left(N\left(\mb\right),\Sigma\left(\mb\right),D_{a}\right) \arrow[r, "Forget"] \arrow[d, "i"] & \mg\left(N\left(\mb\right),\Sigma\left(\mb\right)\right) \arrow[r] \arrow[d, "i"] & 1 \\
  1 \arrow[r] & \pi_{1}\left(\Sigma\left(\mb\right)\right) \arrow[r, "Push"] & Mod\left(\Sigma\left(\mb\right),D_{a}\right) \arrow[r, "Forget"] & Mod\left(\Sigma\left(\mb\right)\right) \arrow[r] & 1
\end{tikzcd}
$$where $i$ denotes the inclusion map. The \emph{Birman exact sequence} provides a description for the kernel of the $Forget$ map, which asserts that the kernel is generated by the isotopies(of $\Sigma\left(\mb\right)$) that push $D_a$ along a closed path(begins and ends at $D_a$) in $\Sigma\left(\mb\right)$. Note that all such isotopies can be extended to the whole manifold $N\left(\mb\right)$. It follows that $K=\pi_{1}\left(\Sigma\left(\mb\right)\right)$. Since both $\pi_{1}\left(\Sigma\left(\mb\right)\right)$ and $\mg\left(N\left(\mb\right),\Sigma\left(\mb\right)\right)$ are finitely generated(or finitely presented), so is $\mg\left(N\left(\mb\right),\Sigma\left(\mb\right),D_{a}\right)$. Similarly, $\mg\left(N\left(\mb'\right),\Sigma\left(\mb'\right),D_{a}\right)$ is also finitely generated(or finitely presented). Then by the exact sequence (3), we know that $\mi(\rho)$ is finitely generated(or finitely presented). Moreover, by the exact sequence (1), $G_{\vec{\mu}}$ is also finitely generated(or finitely presented).
\end{proof}
% we know that $\mi(\rho)$ is finitely generated(or finitely presented). Moreover, by the exact sequence (1),
%\begin{thm}
 %   Assume the \emph{Heegaard splitting} $N=A\cup_{\Sigma}B$ is a connected sum of the two \emph{Heegaard splittings}, $N_1=A_1\cup_{\Sigma_1}B_1$ and $N_2=A_2\cup_{\Sigma_2}B_2$($N_1\text{ and } N_2$ are closed). Let $S$ be a reducing sphere for $\Sigma$, separating these two \emph{Heegaard splittings}, and $\mu=S\cap\Sigma$. If both the two \emph{Goeritz groups} $\mg\left(N_1,\Sigma_2\right)$ and $\mg\left(N_2,\Sigma_2\right)$ are finitely generated, then so is the stabilizer $G_{\mu}$.
%\end{thm}
With all the above lemmas, we can present the proof of Theorem \ref{thm2} as follows.
\begin{proof}[Proof of Theorem \ref{thm2}]
     By Lemma \ref{ost}, we know that $G_{\vec{\mu}}$ is finitely generated(or finitely presented). Since $G_{\vec{\mu}}$ is a subgroup of $G_{\mu}$ with index at most two, $G_{\mu}$ is also finitely generated(or finitely presented).

\end{proof}
The genus at most two \emph{Goeritz groups} for lens spaces or their connected sum are all finitely generated in \cite{cho_genus-two_2013,cho_connected_2016,cho_mapping_2019}. 
Then by Theorem \ref{thm2}, the stabilizer $G_{\mu_{i}}$, which is exactly the group $\mh_i$, is finitely generated. By Theorem \ref{th}, it follows that $\mg\left(N, \Sigma\right)$ is finitely generated. So we complete the proof of Theorem \ref{finitely}.

\bibliographystyle{alpha}

\end{document}